\documentclass[journal,twoside,web]{ieeecolor}

\usepackage{generic}
\usepackage{cite}
\usepackage{amsmath,amssymb,amsfonts}
\usepackage{algorithmic}
\usepackage{graphicx}
\usepackage{textcomp}

\usepackage{graphicx}
\usepackage{subfigure}
\usepackage{epsfig} 
\usepackage{cite}
\usepackage{lcsys}
\usepackage{comment}

\newtheorem{theorem}{Theorem}[section]
\newtheorem{lemma}[theorem]{Lemma}

\newtheorem{corollary}[theorem]{Corollary}
\newtheorem{definition}[theorem]{Definition}

\newtheorem{assumption}[theorem]{Assumption}

\begin{document}

\def\BibTeX{{\rm B\kern-.05em{\sc i\kern-.025em b}\kern-.08em
    T\kern-.1667em\lower.7ex\hbox{E}\kern-.125emX}}
\markboth{\journalname, VOL. XX, NO. XX, XXXX 2017}
{Author \MakeLowercase{\textit{et al.}}: Preparation of Papers for IEEE Control Systems Letters (August 2022)}

\title{Finite-Time Optimization via Scaled Gradient-Momentum Flows}

\author{
Yu Zhou$^{a}$,
Mengmou Li$^{a}$, Masaaki Nagahara$^{a}$ 
\thanks{$^{a}$Graduate School of Advanced Science and Engineering, Hiroshima University, Japan; \texttt{yuzhou@hiroshima-u.ac.jp}, \texttt{mengmou@hiroshima-u.ac.jp}, \texttt{nagam@hiroshima-u.ac.jp}.}
\thanks{This work is supported by JST ASPIRE Project Grant Number JPMJAP2402.}
}

\maketitle
\thispagestyle{empty}

\begin{abstract}
In this paper, we develop a scaled gradient–momentum framework for continuous-time optimization that achieves global finite-time convergence. A state-dependent scaling mechanism is introduced to enable classical dynamics, such as Heavy-Ball-type and proportional–integral (PI)-type flows, to attain finite-time convergence. We establish explicit conditions that bridge the gradient-dominance property of the objective function and finite-time stability of the proposed scaled dynamics. Numerical experiments validate the theoretical results.
\end{abstract}


\section{Introduction}
\label{sec:introduction}

In recent years, continuous-time optimization has attracted increasing attention as a framework that connects optimization algorithms with dynamical systems and control theory. By modeling algorithms as differential equations, it enables systematic analysis and design using tools from nonlinear systems (see, e.g., \cite{su2016:JMLR, wilson2016lyapunov, kovachki2021:JMLR, shi2022:MP, franca2018:ICML}). It also opens new directions for optimization design via feedback control mechanisms (see, e.g., \cite{chen2024:NC, hauswirth2024:Arc}). 

In the continuous-time setting, finite-time stability has attracted increasing attention, as it guarantees convergence to the optimal equilibrium within a finite settling time \cite{bhat2000:Siam_JCO, diana2026:Aut, aal2025finite, ozaslan2024:CDC}. 
For first-order dynamics, various finite-time gradient flows have been proposed, including normalized and sign-based methods \cite{cortes2006:Aut}, as well as extensions via state-dependent scaling \cite{romero2020:ICML, ozaslan2024:CDC}. 
The normalized gradient can also exhibit improved escape behavior from saddle points \cite{murray2019:tac}. 
However, many practical optimization algorithms are inherently second-order, for which finite-time design remains less understood.

When a finite-time continuous optimization algorithm involves a second-order differential equation, many existing works formulate finite-time optimization as a control design problem, treating the gradient as an external input (see, e.g., \cite{aal2025finite}, \cite{texis2025:ECC}, \cite{rios2026:Siam}).
While this control-oriented formulation allows one to leverage established tools from nonlinear control theory, it can also complicate the design and analysis. In fact, optimization fundamentally differs from classical control problems in that it is not constrained by the intrinsic structural properties of a physical plant. Instead, the objective function itself determines the geometry of the induced dynamics. This structural distinction calls for design momentum-driven optimization dynamics, instead of a straightforward adaptation of classical control frameworks.

Finite-time stability is closely related to scaling behavior (see., e.g., \cite{bhat_etal_2005_MCSS}, \cite{denis2020:book}). 
In classical control design, scaling laws must respect the intrinsic 
dynamics and structural properties of the underlying plant. 
In contrast, optimization dynamics are algorithmically constructed rather than physically imposed, which provides significantly greater flexibility in modifying the vector field. 
This flexibility allows the incorporation of state-dependent scaling mechanisms that can be systematically designed to induce finite-time convergence.

Another important concern in finite-time optimization is that the scaling behavior 
may be intrinsically related to the objective function itself. 
However, this relationship has received relatively limited attention 
in the literature. In \cite{romero2020:ICML}, the authors investigate 
the scaling behavior of finite-time gradient flows in connection with 
gradient domination properties of the objective function, 
i.e., the growth rate of the gradient relative to the objective gap. 
When momentum is incorporated, however, the interaction between scaling mechanisms and gradient domination becomes more intricate.

These observations motivate our development of a scaled gradient–momentum flow that achieves finite-time stability. The main contributions of this paper are summarized as follows.

We propose a continuous-time optimization dynamics, called the scaled gradient–momentum flow, which systematically incorporates a state-dependent scaling operator into a second-order gradient–momentum differential equation. The proposed framework systematically scales classical continuous-time Heavy-Ball and PI-type dynamics, thereby achieving finite-time convergence.
Furthermore, we establish a rigorous characterization of the relation between the state-dependent scaling mechanism and the objective’s gradient-dominance property. By deriving explicit matching conditions between the scaling exponent and the gradient-dominance order, we establish sufficient conditions for global finite-time stability in terms of the objective’s local geometry.

The remainder of this paper is organized as follows. Section \ref{sec:preliminaries} introduces the preliminary mathematical tools on finite-time stability and gradient-based optimization. In Section \ref{sec:MainResults}, we detail the proposed scaled gradient-momentum flow, with main stability results. Finally, Section \ref{sec:simulation} provides numerical simulations to validate the theoretical findings.

\section{Preliminaries}
\label{sec:preliminaries}
 For a vector $x \in \mathbb{R}^n$, $\|x\|$ denotes the Euclidean norm.
For a symmetric positive definite matrix $P \in \mathbb{R}^{n \times n}$. The smallest and largest eigenvalues of $P$ are denoted by
$\lambda_{\min}(P)$ and $\lambda_{\max}(P)$, respectively.
The identity matrix of dimension $n$ is denoted by $I_n$.

Consider the nonlinear system
\begin{equation}\label{eq:sys}
    \dot{x} = \phi(x), \qquad x \in \mathbb{R}^n,
\end{equation}
where $\phi : \mathbb{R}^n \to \mathbb{R}^n$ is locally Lipschitz. A point $x_s \in \mathbb{R}^n$ is called an equilibrium (or stationary point) of \eqref{eq:sys} if $\phi(x_s) = 0$. Equivalently, if $x(0) = x_s$, then $x(t) = x_s$ for all $t \ge 0$. Since $x_s$ is constant, by the change of variables $z = x - x_s$, the equilibrium can be shifted to the origin. 
Hence, without loss of generality, we assume that the equilibrium is at $x=0$.

\begin{definition}[\hspace{1sp}\cite{bhat2000:Siam_JCO}]
 Let $\Phi \subset \mathbb{R}^n$ be a neighborhood of the origin. The origin of \eqref{eq:sys} is said to be finite-time stable, if it is Lyapunov stable, and for any initial condition $x_0 \in \Phi \setminus \{\boldsymbol{0}\}$, the trajectory reaches the origin in finite time $T(x_0) < \infty$, i.e., $\|x(t)\| = 0$ for all $t \geq T(x_0)$. 
\end{definition}

The origin is said to be globally finite-time stable if $\Phi = \mathbb{R}^n$.
In the context of continuous-time optimization, 
if the optimal solution is a finite-time stable equilibrium 
of the associated dynamical system, then the corresponding optimization dynamics is said to achieve finite-time convergence. 
In this case, we refer to the system as a 
finite-time continuous-time optimization method.

\begin{theorem}[\hspace{1sp}\cite{bhat2000:Siam_JCO}]\label{thm:finite_time_stability}
Suppose there exists a continuous, positive-definite function $V: \mathcal{D} \to \mathbb{R}$ that is $C^1$ on $\mathcal{D} \setminus \{\boldsymbol{0}\}$,
\[
\psi_1(\|x\|)\le V\le \psi_2(\|x\|)
\]
where $\mathcal{D} \subset \mathbb{R}^n$ is an open neighborhood of the origin, and $\psi_1, \psi_2\in \mathcal{K}_{\infty}$. If there exist constants $c > 0$ and $a \in (0,1)$ such that the time derivative of $V$ along the trajectories of the system satisfies:
\begin{equation}\label{eq:ft_condition}
\dot{V}(x) + c (V (x) )^a\leq 0, \quad \forall x \in \mathcal{D} \setminus \{\boldsymbol{0}\},
\end{equation}
then, the origin is a finite-time stable equilibrium. Moreover, for any initial condition $x_0 \in \mathcal{D}$, the settling time $T(x_0)$ required to reach the origin satisfies: $T(x_0) \leq \frac{V(x_0)^{1-a}}{c(1-a)}$.
\end{theorem}

We now recall several standard structural conditions used in the convergence analysis of gradient-based methods.

\begin{definition}
A differentiable function $f : \mathbb{R}^n \to \mathbb{R}$ 
is said to satisfy the \emph{Polyak--\L{}ojasiewicz (PL) inequality} 
if there exists $\mu>0$ such that $\frac{1}{2}\|\nabla f(x)\|^2 
\ge 
\mu \bigl(f(x)-f^\star\bigr)$, $\forall x \in \mathbb{R}^n$,
where $f^\star := \inf_{x \in \mathbb{R}^n} f(x)$.
\end{definition}

\begin{definition}[Strong Convexity]
A continuously differentiable function 
$f : \mathbb{R}^n \to \mathbb{R}$ 
is said to be $\mu$-strongly convex if there exists $\mu>0$ such that
\[
f(y) 
\ge 
f(x) 
+ \langle \nabla f(x), y - x \rangle
+ \frac{\mu}{2}\|y - x\|^2,
\quad \forall x,y \in \mathbb{R}^n.
\]
\end{definition}

To study scaling effects and more general growth behavior, 
we introduce the following notion.

\begin{definition} [\hspace{1sp}\cite{wibisono2016:PANS}]
Let $f : \mathbb{R}^n \to \mathbb{R}$ be continuously differentiable and 
let $x^\star$ be a local minimizer with $f^\star=f(x^\star)$. 
The function $f$ is said to be \emph{$\mu$-gradient dominated of order $p\in(1,\infty)$} 
in a neighborhood $\mathcal{D}$ of $x^\star$ if there exists $\mu>0$ such that
\begin{equation}\label{eq:grad_domination}
\tfrac{p-1}{p}\|\nabla f(x)\|^{\frac{p}{p-1}}
\ge
\mu^{\frac{1}{p-1}} \bigl(f(x)-f^\star\bigr),
\quad \forall x \in \mathcal{D}.
\end{equation}
If \eqref{eq:grad_domination} holds for all $x\in\mathbb{R}^n$, 
then $f$ is said to be \emph{globally $\mu$-gradient dominated of order $p$}.
\end{definition}

The notion of $\mu$-gradient domination generalizes several standard 
conditions in optimization.

\begin{itemize}
    \item When $p=2$, condition \eqref{eq:grad_domination} reduces to
    the PL inequality.

    \item If $f$ is $\mu$-strongly convex, then it satisfies the PL 
    inequality with the same constant $\mu$. Hence, strong convexity 
    implies gradient domination with $p=2$.

    \item For general $p>1$, $\mu$-gradient domination allows growth rates 
    different from quadratic and does not require convexity. 
    Therefore, it generalizes the PL condition.
\end{itemize}

We conclude this section with a local power inequality 
that will be repeatedly used in subsequent stability analysis.

\begin{lemma}\label{lem:local_power_bound}
Let $a \ge 1$. For any $\delta > 0$, there exists a constant $C>0$ such that for all
$x,y \in [0,\delta]$, the following inequality holds:
$(x+y)^a \le C \left(x^a + y\right)$.
\end{lemma}

\begin{proof}
Let $x,y \in [0,\delta]$ with $\delta>0$ fixed. 
Since the function $q \mapsto q^a$ is convex on $\mathbb{R}_+$ for $a \ge 1$, we have
\begin{equation}
(x+y)^a \le 2^{a-1}\left(x^a + y^a\right).
\end{equation}
Moreover, for $y \in [0,\delta]$ and $a \ge 1$, we have $y^a = y^{a-1} y \le \delta^{a-1} y$.
Therefore, $x^a + y^a \le x^a + \delta^{a-1} y$.
Combining the above inequalities yields
\[
(x+y)^a 
\le 2^{a-1} \left(x^a + \delta^{a-1} y\right)
\le 2^{a-1} \max\{1,\delta^{a-1}\} (x^a + y).
\]
Thus the desired inequality holds with $C = 2^{a-1} \max\{1,\delta^{a-1}\}$.
\end{proof}

 .

\section{Main results}
\label{sec:MainResults}
We aim to solve the following unconstrained optimization problem:
\[
\min_{\theta\in \mathbb{R}^n} f(\theta),
\]
where $f: \mathbb{R}^n \to \mathbb{R}$ is a smooth and differentiable objective function.
Let $\theta^*$ denote the optimal solution and $f^*$ the optimal value.

\begin{assumption}\label{assump:smooth}
The function $f:\mathbb{R}^n \to \mathbb{R}$ is twice continuously differentiable
and has $L$-Lipschitz continuous gradient, i.e., $\|\nabla f(x) - \nabla f(y)\| \le L \|x - y\|$, $\forall x,y \in \mathbb{R}^n$.
Equivalently, the Hessian $\nabla^2 f(x)$ exists and is uniformly bounded as $\|\nabla^2 f(x)\| \le L$, $\forall x \in \mathbb{R}^n$.
\end{assumption}

The objective is to design a continuous-time dynamical system whose trajectories converge to the minimizer $\theta^*$, and to characterize the convergence rate under appropriate structural conditions on $f$.

\subsection{Scaled gradient-momentum flow}
We propose the so-called scaled gradient-momentum flow governed by the following  dynamical system:
\begin{equation}\label{eq:dyn_HomM}
    \begin{aligned}
        \dot{\theta} &=  \|z\|^{\alpha}\left(-(1-\beta)\nabla f(\theta)  +\beta v\right),\\
        \dot{v} &= -\kappa\|z\|^{\alpha} \left(\gamma\nabla f(\theta) + (1-\gamma) v\right),
    \end{aligned} 
\end{equation}
where $v \in \mathbb{R}^n$, 
$z  = [(\nabla f(\theta
))^\top, v^\top]^\top \in \mathbb{R}^2$, $\alpha \in \mathbb{R}$, $\beta\in (0,1]$, $\gamma\in (0,1]$.
The parameter $\kappa > 0$ acts as a time-scaling factor for momentum dynamics.

The exponent $\alpha$ introduces a state-dependent scaling mechanism through the weighted norm $\|z\|^\alpha$.
In particular, when $\alpha = 0$, system~\eqref{eq:dyn_HomM} reduces to a classical gradient--momentum structure. To interpret this property formally, recall that a dynamical system with input
\[
\dot{x} = \phi(x,u)
\]
is said to be homogeneous of degree $d \in \mathbb{R}$ with respect to the standard dilation
$\delta_x = \lambda x$, $\delta_u=\lambda u$
if
\[
\phi(\lambda x, \lambda u) = \lambda^{d+1}\phi(x,u),
\quad \forall \lambda > 0.
\]
By interpreting the gradient $\nabla f(\theta)$ as an input-like signal and examining 
the scaling structure of the term $\|z\|^\alpha z$, 
we observe that under the dilation $(\theta,v) \mapsto (\lambda \theta, \lambda v)$,
the vector field of \eqref{eq:dyn_HomM} scales as $\lambda^{1+\alpha}$.
Hence, the system is homogeneous of degree $d=\alpha$.
In control system design, an asymptotically stable homogeneous system with negative degree is finite-time stable  (see, e.g., \cite{bhat_etal_2005_MCSS},\cite{grune2000SIAM}). However, since the control input is given by the gradient, stability inherently depends on the properties of the objective function. This paper aims to investigate how the scaling parameter (i.e., $\alpha$) and the growth behavior of the objective function (characterized by the gradient-dominance order) jointly determine finite-time stability.

The pair $(\beta,\gamma)$ determines the structure of the gradient-momentum flow.
The parameter $\beta$ determines the relative contribution of the gradient and the velocity in the position update, and the parameter $\gamma$ regulates the alignment of the velocity with the negative gradient.
\begin{itemize}
    \item If $\beta = 1$, it can be interpreted as a state-scaled heavy-ball-type system
\begin{equation}\label{eq:scaled_HB}
        \begin{aligned}
            \dot{\theta} = \|z\|^\alpha v,\
            \dot{v} = - \|z\|^\alpha 
            \big( \gamma \nabla f(\theta) + (1-\gamma) v \big),
        \end{aligned}
    \end{equation}

      \item If $\gamma = 1$, it has a proportional–integral (PI)-type structure in which the velocity integrates gradient information without explicit damping
    \begin{equation}  \dot{\theta} = \|z\|^\alpha(- (1\!-\!\beta)\nabla f(\theta) +\beta v),\ \dot{v} = - \|z\|^\alpha \nabla f(\theta) .  \end{equation}
\end{itemize}

Therefore, the scaled gradient-momentum flow \eqref{eq:dyn_HomM} provides a unified continuous-time framework that interpolates between heavy-ball dynamics, and PI-type structures. The state-dependent scaling term $\|z\|^\alpha$ enables the incorporation of homogeneity-based finite-time convergence mechanisms within a momentum-accelerated optimization architecture.

If $\gamma = 1$ and $\beta = 1$, the system reduces to
\begin{equation}
\begin{aligned}
\dot{\theta} = \|z\|^{\alpha} v,\
\dot{v} = - \|z\|^{\alpha}\nabla f(\theta).
\end{aligned}
\end{equation} 
When $\alpha = 0$, the dynamics become
\[
\dot{\theta} = v, 
\qquad 
\dot{v} = -\kappa \nabla f(\theta),
\]
which admit the energy function $H(\theta,v) = \frac{1}{2}\|v\|^2 + \kappa f(\theta)$.
A direct computation yields $\dot H=v^\top \dot v+\kappa\nabla f(\theta)^\top \dot{\theta}=0$, showing that the energy is conserved along trajectories. If the initial condition satisfies 
$H(\theta(0),v(0)) = h > 0$, 
then $H(\theta(t),v(t)) = h$ for all $t \ge 0$. 
Hence, the trajectory remains on the invariant level set 
$\{(\theta,v): H(\theta,v)=h\}$ 
and cannot converge to the equilibrium point 
$[(\theta^\star)^\top,\boldsymbol{0}^\top]^\top$, where $H=0$. Therefore, the equilibrium is not asymptotically stable. 
The system is purely conservative and admits no dissipation mechanism. 
Consequently, we restrict our analysis to the case $\beta\gamma < 1$, 
where dissipative effects are present.

\subsection{Stability analysis}
In this section, we investigate the convergence properties of the proposed scaled gradient--momentum flow.
In this section, we analyze the convergence properties of the proposed scaled gradient--momentum flow.
We first establish global asymptotic stability under a mild condition on the objective function.

\begin{lemma}\label{lem:asy_stability}
    Assuming $\theta^\star$ is the unique stationary point of $f(\theta)$ such that $\nabla f(\theta) = \boldsymbol{0}$ if and only if $\theta = \theta^\star$. Then, for $ \beta\gamma<1$, the equilibrium $ [(\theta^\star)^\top, \boldsymbol{0}^\top]^\top$ of \eqref{eq:dyn_HomM} is globally asymptotically stable.
\end{lemma}

\begin{proof}
  
Let us consider the Lyapunov candidate as follows
\begin{equation}\label{eq:V}
    V= f(\theta)-f^\star + \tfrac{\beta}{2 \gamma \kappa} v^\top v.
\end{equation}
Then, the derivative of $V$ along the dynamic \eqref{eq:dyn_HomM} has
\begin{equation}
    \begin{aligned}
        \dot{V} =& \ \nabla f_{\theta}^\top\cdot(-(1-\beta)\|z\|^{\alpha}\nabla f_{\theta} +\beta \|z\|^{\alpha}v)\\
        &\ + \tfrac{\beta}{\gamma}v^\top (-\gamma \|z\|^{\alpha}\nabla f_{\theta} - (1-\gamma) \|z\|^{\alpha}v)\\
        =&\  -\|z\|^{\alpha}(1-\beta) \nabla f_{\theta}^\top  \nabla f_{\theta} -\|z\|^{\alpha} \tfrac{\beta(1-\gamma)}{\gamma}v^\top  v,
    \end{aligned}
\end{equation}
where $\nabla f_\theta:= \nabla f(\theta)$.
Hence, for positive $\beta\in (0,1)$ and $\gamma\in (0,1)$,  the equilibrium $ [(\theta^\star)^\top, \boldsymbol{0}^\top]^\top$ is globally asymptotically stable. 

When $\beta=1$ or $\gamma =1$, one has
\begin{equation}
    \dot{V} = \left\{
    \begin{aligned}
         &-\|z\|^{\alpha}(1-\beta) \nabla f_{\theta}^\top  \nabla f_{\theta}, \  & \gamma=1,\\
         &-\|z\|^{\alpha} \tfrac{\beta(1-\gamma)}{\gamma}v^\top  v, \ & \beta=1,
    \end{aligned}
    \right.
\end{equation}
 where $\dot{V}$ is only negative semi-definite, we apply LaSalle's Invariance Principle. Let $\mathcal{S} = \{(\theta, v) \mid \dot{V} = 0\}$ be the set where the derivative vanishes.
For the case $\gamma=1$: $\dot{V} = 0 \implies \nabla f(\theta) = \boldsymbol{0}$. Since $\theta^\star$ is the unique stationary point, this implies $\theta = \theta^\star$. For a trajectory to remain in $\mathcal{S}$, its derivative must satisfy $\dot{\theta} = \boldsymbol{0}$. Substituting $\theta = \theta^\star$ into the $\theta$-dynamics: $\dot{\theta} = \|z\|^\alpha (\beta v) = \boldsymbol{0} \implies v = \boldsymbol{0}$. Thus, the largest invariant set in $\mathcal{S}$ is the equilibrium point $[(\theta^\star)^\top,\boldsymbol{0}^\top]^\top$.

For the  case $\beta=1$: $\dot{V} = 0 \implies v = \boldsymbol{0}$. For a trajectory to stay in $\mathcal{S}$, we require $\dot{v} = \boldsymbol{0}$. Substituting $v=\boldsymbol{0}$ into the $v$-dynamics:
$
\dot{v} = -\kappa \gamma \|z\|^\alpha \nabla f(\theta) = \boldsymbol{0} \implies \nabla f(\theta) = \boldsymbol{0}
$. Since $\nabla f(\theta^\star) = \boldsymbol{0}$ uniquely, the largest invariant set is $[(\theta^\star)^\top,\boldsymbol{0}^\top]^\top$. According to Lasalle's invariance principle (see, e.g., \cite{khalil2002:book}), the $[(\theta^\star)^\top,\boldsymbol{0}^\top]^\top$ is globally asymptotically stable.
\end{proof}

The state-dependent scaling factor $\|z\|^{\alpha}$ only modifies the magnitude of the vector field without altering its direction. 
Consequently, the global asymptotic stability of the corresponding unscaled gradient--momentum dynamics is preserved under such scaling. 
We now investigate how this scaling mechanism, combined with the gradient-dominance property, leads to finite-time stability.

\begin{theorem}\label{thm:main}
    Let $f: \mathbb{R}^n \to \mathbb{R}$ satisfy Assumption \ref{assump:smooth} and be $\mu$-gradient dominated of order $p\in(1,4]$. For  $-1\le \alpha <\min\left\{\tfrac{2(2-p)}{p},0\right\}$, the equilibrium $ [(\theta^\star)^\top, \boldsymbol{0}^\top]^\top$ of the system \eqref{eq:dyn_HomM} is globally finite-time stable if any of the following conditions holds:
\begin{enumerate}
\item $\beta \in (0,1)$ and $\gamma \in (0,1)$;
\item $\beta = 1, \gamma \in (0,1)$, and the Hessian satisfies $\nabla^2 f(\theta) \succ 0, \forall \theta \in \mathbb{R}^n$;
\item $\gamma = 1, \beta \in (0,1)$, and the Hessian satisfies $\nabla^2 f(\theta) \succeq m I_n$ for some $m > 0, \forall \theta \in \mathbb{R}^n$.
    \end{enumerate}
\end{theorem}

\begin{proof}
We first show the finite-time stability for $\beta\in (0,1)$ and $\gamma\in (0,1)$. Given that $f$ is $L$-smooth (with $L$-Lipschitz gradient), the following standard bound holds:
\begin{equation}\label{eq:fgap-grad-bound}
\tfrac{1}{2L}\|\nabla f(\theta)\|^2\le f(\theta)-f^\star.
\end{equation}

Then, the lower bound of $V$ has
\begin{equation}
\begin{aligned}
         V&\ge\tfrac{1}{2L}\|\nabla f_\theta\|^2 +    \tfrac{\beta}{2\gamma\kappa}\|v\|^2
         \ge  c_1\|z\|^2.
\end{aligned}
 \label{neq:v_lower}
\end{equation}
where $c_1:=\min\left\{\tfrac{1}{2L}, \tfrac{\beta}{2\gamma\kappa}\right\}$.
By the gradient dominance, we have
\[
\|\nabla f_\theta\|\ge  b(f_\theta-f^\star)^{\eta},
\]
where $\eta:=\tfrac{p-1}{p}\in (0,1)$, $b:=\left[\tfrac{1}{\eta}\mu^{1/(p-1)}\right]^{\eta}$. Then, we have the upper bounded of $V$:
\begin{equation}\label{eq:v_upper}
    V\le \eta \mu^{\tfrac{1}{1-p}} \|\nabla f_\theta\|^{\tfrac{1}{\eta}} +    \tfrac{\beta}{2\gamma\kappa}\|v\|^2\le c_2(\|\nabla f_\theta\|^{\tfrac{1}{\eta}}+\|v\|^2),
\end{equation}
where $c_2:=\max\left\{\eta \mu^{1/(1-p)}, \tfrac{\beta}{2\gamma\kappa} \right\}$.

On the other hand, the derivative has
\begin{equation*}
        \begin{aligned}
            \dot{V} &\le- (1-\beta)\|z\|^{\alpha} \|\nabla f_\theta\|^2 -\tfrac{\beta(1-\gamma)}{\gamma}\|z\|^{\alpha}v^\top v\\
            &\le -(1-\beta)b \|z\|^{\alpha}  (f-f^\star)^{2\eta}-\tfrac{\beta(1-\gamma)}{\gamma}\|z\|^{\alpha}\|v\|^2,\\
& \le -\tilde{b}\|z\|^{\alpha}\left(
(f-f^\star)^{2\eta}+\|v\|^2
\right),
        \end{aligned}
\end{equation*}
where $\tilde{b}:=\min\left\{(1-\beta)b,\tfrac{\beta(1-\gamma)}{\gamma}\right\}$.

\textbf{Case I:}   $p\in [2,4)$. We have $2\eta\ge 1$. According to Lemma \ref{lem:local_power_bound}, for a  neighborhood of the stationary point $\Omega:=\{z\in \mathbb{R}^{2n}:\|z\|\le \delta, \ \delta>0\}$, there exists a positive constant $C_{\delta}$, such that the following inequality holds:
\begin{equation}\label{eq:estimate_1}
    C_{\delta}((f-f^\star)^{2\eta}+\|v\|^2) \ge \left((f-f^\star)+\|v\|^2\right)^{2\eta}.
\end{equation}
Thus, we obtain for all $z \in \Omega$: 
\begin{equation}\label{eq:d_V_p_1}
    \dot{V}\le -\tfrac{\tilde{b}}{C_\delta}\|z\|^\alpha(f-f^\star + \|v\|^2)^{2\eta}. 
\end{equation}

\textbf{Case II:}  $p\in (1,2]$. We have $2\eta\in (0,1]$. 
Therefore, for a neighborhood of the equilibrium, $\tilde{\Omega}
:=
\{z\in\mathbb{R}^{2n} \mid f-f^\star\le 1\}$, it holds that, 
\begin{equation}\label{eq:estimate_2}
    (f-f^\star)^{2\eta} \ge f-f^\star.
\end{equation}
Thus, according to \eqref{neq:v_lower} and \eqref{eq:v_upper}, for any $z\in \tilde{\Omega}$,  $V $ has
\[
\dot{V}\le -\tilde{b}\|z\|^{\alpha}\left(
(f-f^\star)+\|v\|^2
\right).
\]
Besides, from the lower bound of $V$ in \eqref{neq:v_lower}, for $\alpha<0$, one has $V^{\tfrac{\alpha}{2}}\le(\sqrt{c_1}\|z\|)^\alpha$.
Combining the above estimates, there  exists a neighborhood of the stationary point $\Omega\cap \tilde{\Omega}$ such that
\begin{equation}
    \dot{V}\le \left\{
    \begin{aligned}
    &-\tfrac{\tilde{b}}{C_\delta}c_1^{-\alpha/2}V^{\tfrac{\alpha}{2}+2\eta} , \ &  p\in [2,4),\\
    &-\tilde{b}c_1^{-\alpha/2}V^{\tfrac{\alpha}{2}+1}, \ & p\in (1,2].
    \end{aligned}\right.
\end{equation}
Since  $-1\le \alpha <\min\left\{\tfrac{2(2-p)}{p},0\right\}$, it has $\tfrac{\alpha}{2}+2\eta\in (0,1)$ and $\tfrac{\alpha}{2}+1\in (0,1)$, by Theorem~\ref{thm:finite_time_stability},  for both cases, the \eqref{eq:dyn_HomM}  locally finite-time converge to the stationary point.
On the other hand, by global asymptotic stability in Lemma~\ref{lem:asy_stability}, every trajectory enters $\Omega\cap \tilde{\Omega}$ in finite time. Therefore, the equilibrium is globally finite-time stable.

Now, we consider the parameters $\beta=1$ or $\gamma=1$, in this case, $\dot{V}$ becomes negative semi-definite (weakly dissipative). 
To establish strict dissipativity and thus finite-time stability, for the case of $\beta=1$ or $\gamma=1$, we consider a new Lyapunov candidate
\[
\tilde{V} = V - \epsilon v^\top\nabla f_\theta.
\]
We firstly show that for sufficiently small $\epsilon>0$,  $\tilde{V}$ is well-posed.
By \eqref{neq:v_lower}, we have the lower bound of $\tilde{V}$
\[
\tilde{V}\ge \frac{1}{2}z^\top W z, \quad W = \left[\begin{smallmatrix}
    1/L I_n & -\epsilon I_n\\
    -\epsilon I_n & \beta/(\gamma\kappa) I_n
\end{smallmatrix}\right].
\]
By the Schur complement condition, for sufficiently small $\epsilon < \sqrt{\tfrac{\beta}{\gamma\kappa L}}$,  $W$ is positive definite.

The derivative of $\tilde{V}$ has
\begin{equation}
    \begin{aligned}
        \dot{\tilde{V}} 
        =& \  \dot{V} - \epsilon\kappa \|z\|^\alpha( \gamma \nabla f_\theta + (1-\gamma) v)^\top \nabla f_\theta\\ 
        &- \epsilon\|z\|^\alpha v^\top \nabla^2 f_\theta (-(1-\beta)\nabla f_\theta +\beta v)
    \end{aligned}
\end{equation}
For $\beta=1$, it has
\begin{equation}
    \dot{\tilde{V}} = \dot{V} \!-\!\epsilon \|z\|^\alpha\left[\kappa( \gamma \nabla f_\theta + (1-\gamma) v)^\top \nabla f_\theta -  v^\top\nabla^2 f_\theta  v\right].
\end{equation}

For $\beta=1$ and $\gamma\neq 1$, since $\nabla^2 f_\theta$ is positive definite, we have
\begin{equation}
\!\!\!\dot{\tilde{V}}\!\le -\epsilon\kappa\|z\|^\alpha\left[\tfrac{\beta(1-\gamma)\|v\|^2}{\epsilon\kappa\gamma}\!+\!\gamma\|\nabla f_\theta\|^2 \!+\!(1\!-\!\gamma)v^\top \nabla f_\theta\right]\!,
\end{equation}
equivalent to
\begin{equation}
    \dot{\tilde{V}}\le -\|z\|^\alpha z^\top W_1 z, \ W_1\!=\!\left[\begin{smallmatrix}
    \epsilon\kappa \gamma I_n& \epsilon\kappa(1-\gamma)/2 I_n\\ \epsilon\kappa(1-\gamma)/2 I_n & \tfrac{\beta(1-\gamma)}{\gamma} I_n
\end{smallmatrix}\right].
\end{equation}
By Schur complemry, for sufficiently small $\epsilon\le\frac{\beta}{\kappa(1-\gamma)} $, $W_1$ is positive definite. 
Then, we have 
\begin{equation}
    \begin{aligned}
        \dot{V}\le -\lambda_{\min}(W_1)\|z\|^\alpha(\|\nabla f_\theta\|^2+\|v\|^2)
    \end{aligned}
\end{equation}
Using the same estimation in 
\eqref{eq:estimate_1}, \eqref{eq:estimate_2}, and $\tilde{V}\ge \tfrac{{\lambda_{\min}(W)}}{2}\|z\|^2$, we obtain the following local inequality:
\begin{equation}
    \!\!\!\!\dot{\tilde{V}}\!\le\! \left\{
    \begin{aligned}
    &\!-\tfrac{\lambda_{\min}(W_1)}{C_\delta}\left(\tfrac{2}{\lambda_{\min}(W)}\right)^{\alpha/2}\tilde{V}^{\tfrac{\alpha}{2}+2\eta} ,  &  p\in [2,4),\\
    &\!-\lambda_{\min}(W_1)\left(\tfrac{2}{\lambda_{\min}(W)}\right)^{\alpha/2}\tilde{V}^{\tfrac{\alpha}{2}+1},  & p\in (1,2].
    \end{aligned}\right.
\end{equation}
Local finite-time stability then follows directly from Theorem~\ref{thm:finite_time_stability}. Since global asymptotic stability has already been established, we conclude that the equilibrium is globally finite-time stable.

For $\gamma=1$ and $\beta\neq 1$, we have
\begin{equation}
\begin{aligned}
\dot{\tilde{V}}=& \ \dot{V}\! -\!\|z\|^\alpha(\epsilon\kappa  \|\nabla f_\theta\|^{2}- \epsilon v^\top \nabla^2 f_\theta (-(1\!-\!\beta)\nabla f_\theta +\beta v))\\
       \le& -\|z\|^\alpha\left[(1-\beta+\epsilon \kappa)\|\nabla f_\theta\|^2+ \epsilon\beta v^\top \nabla^2 f_\theta v \right]\\
    &-\|z\|^\alpha\epsilon(1-\beta)v^\top\nabla^2 f_\theta \nabla f_\theta.
\end{aligned}
\end{equation}
The lower bound on the Hessian implies $v^\top \nabla^2 f_\theta v\ge m  \|v\|^2$.
Moreover, since the Hessian is bounded ( $\|\nabla^2 f_\theta\| \le L$), one has $
|v^\top \nabla^2 f_\theta \nabla f_\theta| \le \|v\| \|\nabla^2 f_\theta \nabla f_\theta\| \le L \|v\| \|\nabla f_\theta\|
$.
Now, we apply Young's Inequality with a parameter $\delta > 0$. The standard form $ab \le \frac{1}{2\sigma}a^2 + \frac{\sigma}{2}b^2$ gives us:
$$
L \|v\| \|\nabla f_\theta\| \le \tfrac{L}{2\sigma}\|v\|^2 + \tfrac{L\sigma}{2}\|\nabla f_\theta\|^2
$$
Thus, we have
\begin{equation}
    \dot{\tilde{V}}\le -\|z\|^\alpha z^\top W_2 z,
\end{equation}
where $W_2\!=\!\left[\begin{smallmatrix}
    \left(1\!-\!\beta +\epsilon \kappa - \tfrac{\epsilon L(1-\beta)}{2\sigma}\right)I_n& \boldsymbol{0}\\ \boldsymbol{0} & \left(\epsilon\beta m\!-\! \tfrac{L(1-\beta)\epsilon\sigma}{2}\right) I_n
\end{smallmatrix}\right]$.

By selecting  $\sigma < \frac{L(1-\beta)}{2\beta m}$ there exists a sufficiently small $\epsilon$, such that $W_2$ is positive definite. 
\begin{equation}\label{eq:W_2}
    \dot{\tilde{V}}\le -\lambda_{\min}(W_2)\|z\|^\alpha (\|\nabla f_\theta\|^2+\|v\|^2).
\end{equation}
Applying the same arguments for \eqref{eq:W_2} as in the case $\beta = 1$, $\gamma \neq 1$,  we obtain finite-time stability. 
This completes the proof. 
\end{proof}

The parameter $p$ characterizes the local growth behavior of the objective 
function in a neighborhood of the stationary point. 
When $p \in (1,2)$, the gradient domination condition implies that 
$\|\nabla f(\theta)\|$ remains relatively large compared to 
$f(\theta)-f^\star$ near $\theta^\star$, which strengthens the descent effect 
and facilitates faster convergence.  In contrast, when $p \in [2,4)$, the gradient vanishes more rapidly 
as $\theta \to \theta^\star$, reflecting a flatter local geometry 
and yielding weaker descent dynamics.

From Theorem \ref{thm:main}, we would like to emphasize the following observations.

\begin{itemize}
    \item 
     First, the role of the scaling parameter $\alpha$ depends critically on the dominance order $p$. 
For $p \in [2,4)$, where the flatness of the landscape slows down convergence, a negative scaling exponent $\alpha\in [-1,2(2-p)/p)$ is necessary to compensate for the vanishing gradient and to ensure finite-time convergence with an acceleration effect. 
On the other hand, when $p \in (1,2)$, finite-time stability can be achieved with any negative $\alpha$ due to the intrinsic sharpness of the objective. 
Nevertheless, selecting a larger $|\alpha|$ further accelerates the convergence speed near the optimal.

\item Second, the stability mechanism is structure-dependent.  In the critical configurations $\beta=1$ or $\gamma=1$, the Lyapunov derivative reduces to a weak dissipation condition $\dot V \le 0$. Strict finite-time stability can be restored under additional Hessian conditions that enforce sufficient energy dissipation.
\end{itemize}

While Theorem \ref{thm:main} establishes the finite-time convergence of the objective value and the momentum, it is often desirable to characterize the convergence rate in the parameter space. 
For the specific dominance order $p=2$, the system admits a global finite-time Lyapunov function, which allows for a global time estimate to reach the optimal set. Without loss of generality, we assume that the momentum state is initialized at $v(0)=\mathbf{0}$.
\begin{corollary}
   Suppose $f:\mathbb{R}^n \to \mathbb{R}$ is $\mu$-strongly convex and satisfies Assumption \ref{assump:smooth}. Let the momentum state be initialized at $v(0)=\mathbf{0}$. For $\alpha \in (-1,0)$ and parameters $\beta, \gamma \in (0,1]$ satisfying $\beta\gamma < 1$, the trajectories of the system \eqref{eq:dyn_HomM} converge to the optimal solution $\theta^\star$ in finite time. Specifically, there exists a constant $C > 0$ such that:
\begin{equation}
\|\theta(t) - \theta^\star\| \le C \big(T_s - t\big)^{-1/\alpha}, \quad \forall t \in [0, T_s),
\end{equation}
where $T_s := \tfrac{2(f(0)- f^\star)^{-\alpha/2}}{-\alpha \rho}$ for some constant $\rho>0$. Furthermore, for all $t \ge T_s$, $\|\theta(t) - \theta^\star\| = 0$.
\end{corollary}
\begin{proof}
    For $v(0)=\boldsymbol{0}$, we have $\tilde{V}(0) =V(0) = f(\theta(0))-f(\theta^\star)$. By lipschitz condition, we have $\tilde{V}(0) =V(0)\le \tfrac{L}{2} \|\theta(0)-\theta^\star\|^2$.
    Since $V$ corresponds to the special case $\epsilon=0$ of $\tilde V$,  we restrict the analysis to $\tilde V$ without loss of generality.
    On the other hand, due to strong convexity, one has 
   \begin{equation}
    \tilde{V}\ge \tfrac{1}{2}\lambda_{\min}(W)\|\nabla f_\theta\|^2\ge \tfrac{\mu}{2}\lambda_{\min}(W)\|\theta(t)-\theta^\star\|^2.
   \end{equation}
    According to the proof of the Theorem \ref{thm:main}, there exists some constant $\rho>0$ such that
    \begin{equation}
        \dot{\tilde{V}}(t)\le -\rho \tilde{V}^{\tfrac{\alpha}{2}+1}.
    \end{equation}
    Since $\alpha\in(-1,0)$, we have $\frac{\alpha}{2}+1\in(0,1)$.
Solving this differential inequality yields $\tilde V(t)
\le
\left[
\tilde V^{-\alpha/2}(0)
+
\tfrac{\alpha\rho}{2}t
\right]^{-2/\alpha}$.

The right-hand side reaches zero at 
\[
T_s =
\tfrac{2 \tilde V(0)^{-\alpha/2}}{-\alpha \rho }
 = \tfrac{2(f(0)- f^\star)^{-\alpha/2}}{-\alpha \rho}
\le  \tfrac{2 (L\|\theta(0)- \theta^\star\|)^{-\alpha/2}}{-\alpha \rho}.
\]
 By the upper bound and lower bound of the Lyapunov function, we have
    \[
    \tfrac{\mu\lambda_{\min}(W)\|\theta(t)-\theta^\star\|^2}{2}\le \left[ \tfrac{-\alpha\rho}{2}(T_s-t) \right]^{-\frac{2}{\alpha}}.
    \]
    Identifying $C = \sqrt{\frac{2}{\mu}\lambda_{\min}(W)} (\frac{-\alpha \rho}{2})^{-1/\alpha}$, the proof is complete.
\end{proof}

The scaling parameter $\alpha$ provides a direct mechanism to tune both the decay behavior and the settling time.

\section{Numerical illustration}\label{sec:simulation}
We illustrate the theoretical results using two representative objective functions. We consider the Rosenbrock function
\begin{equation}
f(\theta_1,\theta_2)
=
100(\theta_2-\theta_1^2)^2
+
(1-\theta_1)^2,
\end{equation}
which is nonconvex and admits a unique global minimizer 
$(\theta_1^\star,\theta_2^\star)=(1,1)$. 
This example is used to evaluate the behavior of the proposed dynamics in a nonconvex setting. 
For this function, the gradient-dominance order is $p=2$. 
We demonstrate the effect of the scaling parameter $\alpha\in \{-0.25, -0.5, -0.75\}$, as well as different parameter structures: 1). $\beta=1$, $\gamma \in (0,1)$; 2). $\beta \in (0,1)$, $\gamma=1$; 3). $\beta, \gamma \in (0,1)$.
To examine the influence of the order of gradient dominance, we consider
\begin{equation}\label{eq:example_2}
f(\theta)=\tfrac{1}{p}\|\theta\|^p, 
\end{equation}
whose unique minimizer is $\theta^\star=0$. We select $p=1.5, 2, 3$ for our validation. According to Theorem~\ref{thm:main}, we choose $\alpha=-0.8$ and $\beta=\gamma=0.5$ to guarantee finite-time stability.
 
The numerical results are presented in Figs.~\ref{fig:1} and~\ref{fig:2}. 
From Fig.~\ref{fig:1}, for the Rosenbrock function, larger values of $|\alpha|$ result in a shorter settling time and a faster local convergence rate under the scaled gradient--momentum flow. 
Since the Hessian of the Rosenbrock function is not sign-definite over the entire domain, the strong convexity condition is not globally satisfied. 
As shown in the right subfigure of Fig.~\ref{fig:1}, the configurations with $\beta=1$ or $\gamma=1$ fail to achieve finite-time convergence, which is consistent with Theorem~\ref{thm:main}. 

Figure~\ref{fig:2} illustrates the behavior under different gradient-dominance orders for the objective function \eqref{eq:example_2}. In all cases, the trajectories converge to the optimal solution in finite time. 
Moreover, the order of gradient dominance directly influences the settling time: smaller values of $p$ lead to shorter settling times.

\begin{figure}[h]
    \centering
    \includegraphics[width=0.493\linewidth]{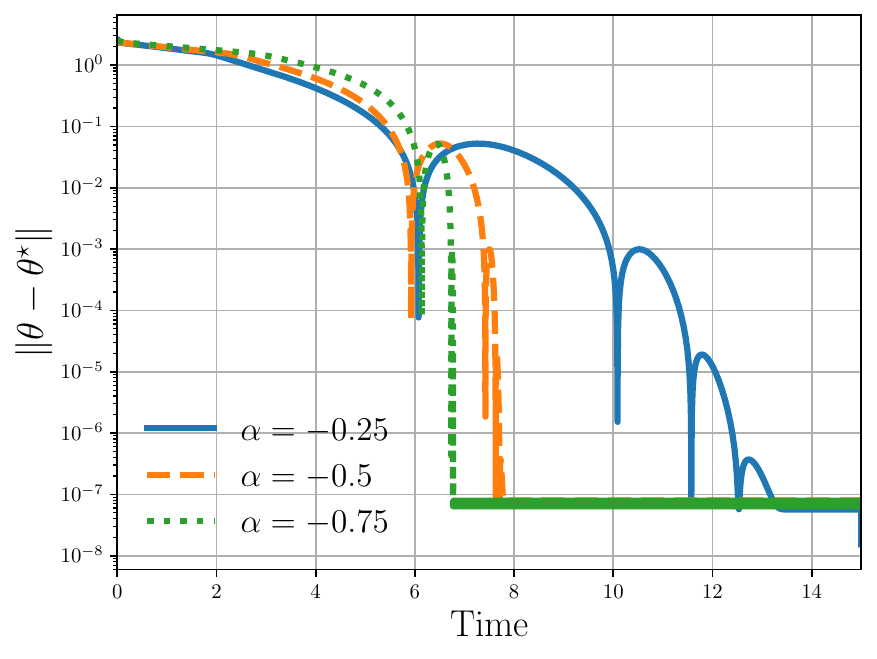}
\includegraphics[width=0.493\linewidth]{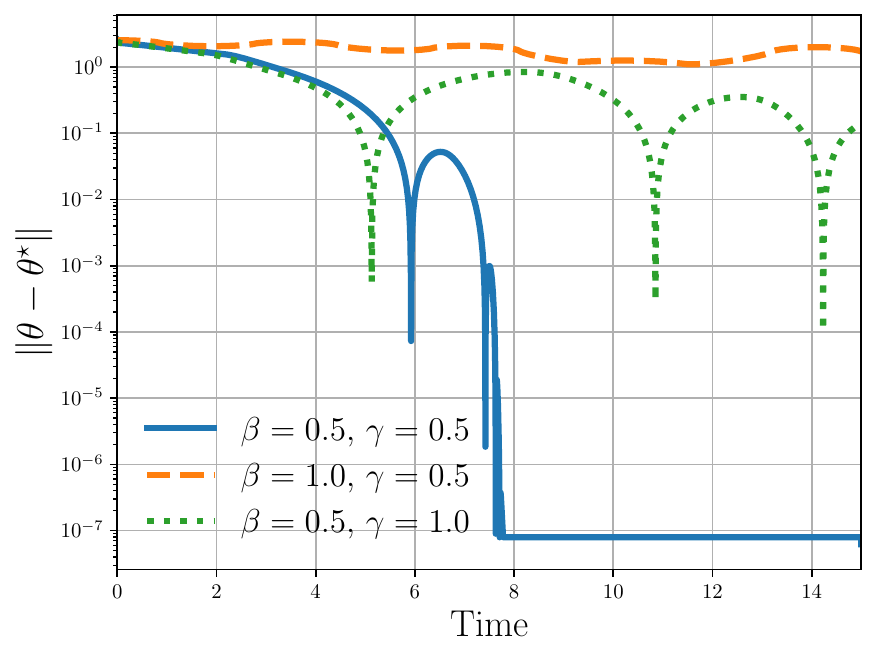}

    \caption{Time evolution of $\|\theta(t)-\theta^\star\|$ for the Rosenbrock function on a logarithmic scale.
Left: comparison of different $\alpha$ values with $\beta=\gamma=0.5$.
Right: comparison of different $(\beta,\gamma)$ pairs with $\alpha=-0.5$.}
    \label{fig:1}
\end{figure}

\begin{figure}[h]
    \centering
    \includegraphics[width=0.493\linewidth]{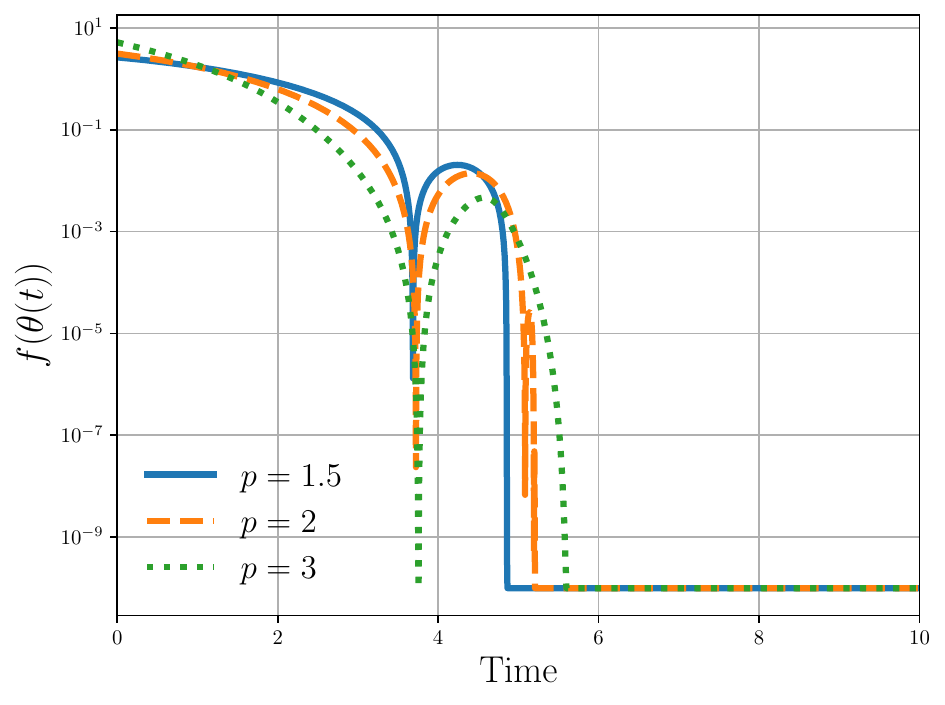}
\includegraphics[width=0.493\linewidth]{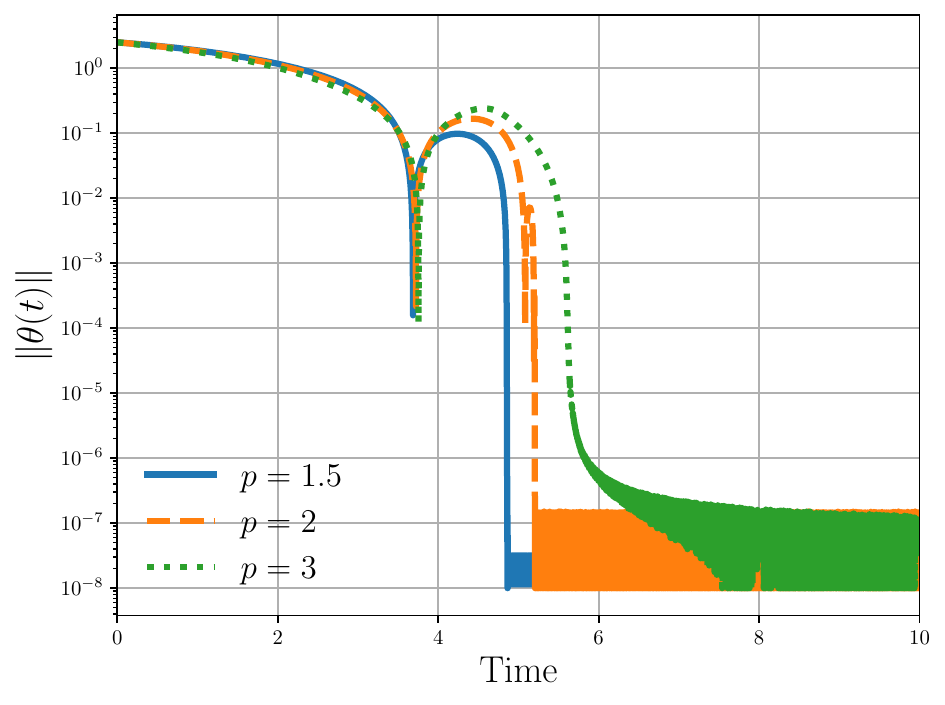}
    
    \caption{Time evolution of the objective value $f(\theta(t))$ (left) and the state norm $\|\theta(t)\|$ (right) for the function~\eqref{eq:example_2} on a logarithmic scale.}
    \label{fig:2}
\end{figure}

\section{Conclusion}
This paper proposed a scaled gradient–momentum flow that achieves global finite-time convergence for unconstrained optimization by extending classical momentum dynamics from asymptotic to finite-time stability through state-dependent scaling. We show that the effect of the scaling parameter $\alpha$ depends critically on the objective’s growth order $p$. For flatter landscapes ($p \in [2,4)$), a sufficiently negative scaling exponent is required to  ensure finite-time convergence. In contrast, for sharper objectives ($p \in (1,2)$), finite-time stability can be achieved for any negative $\alpha$. Moreover, finite-time stability is structure-dependent. In weakly dissipative configurations (e.g., $\beta = 1$ or $\gamma = 1$), strict finite-time convergence requires additional Hessian conditions. Future work will focus on the discretization of the proposed scaled gradient–momentum flow.

\bibliographystyle{IEEEtran}
\bibliography{Ref}

\end{document}